\documentclass[11pt, a4paper]{amsart}
\usepackage{a4wide}
\usepackage{mathrsfs, amssymb,amsmath,url,amsthm}
\usepackage{cite, hyperref}

\usepackage{graphicx}
\usepackage{footnote}

\newcommand{\Z}{{\mathbb Z}}

\newcommand{\alg}[1]{\mathbf{#1}}

\newcommand{\algA}{\alg{A}}

\DeclareMathOperator{\pEq}{\mathsf{pEq}}
\DeclareMathOperator{\pId}{\mathsf{pId}}

\DeclareMathOperator{\Pol}{Pol}

\theoremstyle{plain}

    \newtheorem{thm}{Theorem}[section]
    \newtheorem{theorem}[thm]{Theorem}

    \newtheorem{lemma}[thm]{Lemma}

    \newtheorem{corollary}[thm]{Corollary}

		\newtheorem{question}[thm]{Question}

\theoremstyle{definition}

    \newtheorem{definition}[thm]{Definition}

\theoremstyle{remark}

\title{The equation solvability problem over supernilpotent algebras with Mal'cev term}

\author{Michael Kompatscher}

\address{Department of Algebra, MFF UK,\\
Sokolovska 83, 186 00 Praha 8, Czech Republic\\
\email{michael@logic.at}\\
}
\thanks{
Supported by Charles University Research Centre program No.PRIMUS/SCI/12 and No.UNCE/SCI/022 as well as grant 18-20123S of the Czech Grant Agency (GA\v{C}R). ORCID iD: 0000-0002-0163-6604}


\begin{document}

\begin{abstract}
In 2011 Horv\'ath gave a new proof that the equation solvability problem over finite nilpotent groups and rings is in P. In the same paper he asked whether his proof can be lifted to nilpotent algebras in general. We show that this is in fact possible for supernilpotent algebras with a Mal'cev term. However, we also describe a class of nilpotent, but not supernilpotent algebras with Mal'cev term that have coNP-complete identity checking problems and NP-complete equation solvability problems. This proves that the answer to Horv\'ath's question is negative in general (assuming P$\neq$NP).
\end{abstract}

\maketitle


\section{Introduction}
One of the oldest problems in algebra is to decide whether an equation over a given algebraic structure has a solution. In the last decades this problem has received increasing attention from a computational complexity point of view; in particular for finite algebras the aim is to identify conditions that either imply tractability or hardness of the corresponding equation solvability problem. Many results are known for finite groups and rings, of which several are based on commutator theory. For rings a complexity dichotomy holds: In~\cite{BurrisLawrence-rings} it was shown that the equation solvability problem over non-nilpotent rings is NP-complete, by~\cite{HorvathEqSolvNilp} the problem is in P for nilpotent rings.  

Nilpotency is also a source of tractability in the group case: It was proven in~\cite{goldmannrussell} (and reproven in~\cite{HorvathEqSolvNilp}) that the equation solvability over nilpotent groups is in P. By~\cite{Horvathetal-nonsolvablegroups} non-solvable groups induce NP-complete problems. Furthermore it was shown in \cite{HorvathSzabo-extendedgroups} that every solvable, non-nilpotent group has a polynomial extension, whose equation solvability problem is NP-complete. However it is still open whether a complexity dichotomy like over rings holds. In particular nilpotency does not demark the border between problems in P and NP-complete: By~\cite{HorvathSzabo-A4} the equation solvability over the non-nilpotent group $A_4$ is in P but its extension by the commutator $[\cdot,\cdot]$ has an NP-complete equation solvability problem. More general, meta-abelian groups~\cite{horvath-metaabelian} and semipattern groups~\cite{Attila-semipattern} induce equation solvability problems that are in P, while not necessarily being nilpotent.

Congruence permutable varieties generalize both the varieties of groups and rings and are well-studied in the context of commutator theory. It is hence natural to ask, whether the above dichotomy results can be generalized to all congruence permutable varieties. It was already observed in~\cite{AichingerMudrinski-highercomm} that the identity checking problem for supernilpotent such algebras is in P. Also our results indicate that not nilpotency, but supernilpotency is the right notion to work with: In Section~\ref{section:supernilpotent} we show that the equation solvability problems over finite supernilpotent algebras with Mal'cev term are in P. As a corollary of our proof we obtain a new characterization of supernilpotent algebras in congruence permutable varieties. In Section~\ref{section:nilpotent} we give examples of nilpotent, but not supernilpotent algebras (of infinite type) that have a Mal'cev term and induce coNP-complete identity checking problems and NP-complete equation solvability problems.

The following two subsections provide some necessary preliminary definitions and facts about nilpotent and supernilpotent algebras.

\subsection{Equation solvability and identity checking} \label{sect:eqsolv}
We are going to denote algebras by bold characters and their domain by the corresponding non-bold character (e.g. $\algA$ is an algebra on the set $A$). A \emph{polynomial} over an algebra $\algA$ is a term that is built from variables and elements of $A$ using the operation symbols of $\algA$. We write $\Pol(\algA)$ for the set of all polynomials over $\algA$ and $\Pol_n(\algA)$ for the set of polynomials of arity $n$. We say two algebras on the same domain $A$ are \emph{polynomially equivalent} if their polynomials induce the same operations on $A$.

The \emph{(polynomial) equation solvability problem} over an algebra $\algA$, short $\pEq(\algA)$, asks whether or not two polynomials $f(\bar x)$, $g(\bar x)$ over $\algA$ can attain the same value for some substitution over $\algA$. In other words, for input $f(\bar x)$, $g(\bar x)$ the question is whether $\algA \models \exists \bar x f(\bar x) = g(\bar x)$.

The \emph{(polynomial) identity checking problem} over $\algA$, short $\pId(\algA)$, asks whether a given equation is satisfied under \emph{all} substitution of the variables by elements of $\algA$. So, for two input polynomials $f(\bar x)$, $g(\bar x)$ the question is whether $\algA \models \forall \bar x f(\bar x) = g(\bar x)$. In literature the identity checking problem is sometimes also referred to as \emph{equivalence problem}.

When phrasing $\pEq(\algA)$ and $\pId(\algA)$ as actual computational problems, it is not obvious how to encode the input. However, when we are studying finite algebras of finite type, then an input term can just be encoded by the string defining it. Hence the size of an input polynomial $p(\bar x)$ is proportional to its \emph{length $l(p(\bar x))$}, i.e. the total number of functions, constants and variable symbols it contains.

We remark that this encoding might not be optimal, in the sense that it does not take into account that some expressions might be repeatedly used in the definition of a term. An alternative would be to describe terms by \emph{algebraic circuits}. This approach was suggested by Ross Willard; not much is known for this encoding. Also, for some algebras it makes sense to restrict the input to terms of a certain canonical form (for instance~\cite{SzaboVertesi-rings}, \cite{Horvath-Rings} and \cite{HLW-Sumofmonomialsrings} study the sum of monomials over rings); this is also something we are not considering here. For discussions on the size of term representation in supernilpotent algebras, see also~\cite{AichingerMudrinskiOprsal-length}.

\subsection{Nilpotent algebras with Mal'cev term}
A ternary term $m$ over an algebra $\algA$ is called a \emph{Mal'cev term} if it satisfies $m(y,x,x) = m(x,x,y) = y$ for all $x,y \in A$. It is well-known that an algebra has a Mal'cev term if and only if it is from a congruence permutable variety.
In this section we provide some results on nilpotent and supernilpotent algebras that have a Mal'cev term. For background on commutator theory we refer to~\cite{FreeseMcKenzie-Commutator}, for a survey on higher commutators we refer to~\cite{AichingerMudrinski-highercomm}.

Let $\algA$ be an algebra and $\alpha_1,\ldots,\alpha_n$ be congruence relations of $\algA$. Then $\delta = [\alpha_1,\ldots,\alpha_n]$ denotes the smallest congruence relation of $\algA$ such that for all polynomials $f(\bar x_1,\ldots, \bar x_n)$ and for all tuples $\bar a_1,\bar b_1, \ldots, \bar a_n, \bar b_n$ in $\algA$ with $\bar a_i \equiv_{\alpha_i} \bar b_i$ we have that, $f(\bar a_1, \bar x_2, \ldots, \bar x_n) \equiv_{\delta} f(\bar b_1, \bar x_2 \ldots, \bar x_n)$ for all $(x_2,\ldots,x_n) \in \prod_{i = 2}^n \{a_i,b_i\} \setminus \{(b_2,\ldots, b_n)\}$, implies that $f(\bar a_1, \bar b_2, \ldots, \bar b_n) \equiv_{\delta} f(\bar b_1, \bar b_2 \ldots, \bar b_n)$.

\begin{definition}
Let $\algA$ be an algebra and let $\mathbf{1}_\algA$ denote the total equivalence relation and $\mathbf{0}_\algA$ the identity on $\algA$. Then:
\begin{itemize}
\item $\algA$ is called \emph{nilpotent (of degree $n$)} if $[\mathbf{1}_\algA, \underbrace{[\mathbf{1}_\algA, \ldots, [\mathbf{1}_\algA,[\mathbf{1}_\algA,\mathbf{1}_\algA]] \ldots ]]}_{n}] = \mathbf{0}_\algA$.
\item  $\algA$ is called \emph{supernilpotent (of degree $n$)} if $[\underbrace{\mathbf{1}_\algA,\mathbf{1}_\algA,\ldots,\mathbf{1}_\algA}_{n+1}] = \mathbf{0}_\algA$.
\end{itemize}
\end{definition}

We remark that, for algebras with Mal'cev term, supernilpotency of degree $n$ implies nilpotency of degree $n$. For groups and rings the two notions are equivalent, this is however not true in general. Every nilpotent algebra with Mal'cev term gives rise to loop operations, where a loop is defined as follows:


\begin{definition}
An algebra $\alg{L} = (L,\cdot, \backslash ,/, 0)$ is called a \emph{loop} if for all $x,y \in L$:
\begin{enumerate}
\item $x \backslash (x \cdot y) = y$ and $(y \cdot x) / x = y$
\item $x \cdot (x \backslash y) = y$ and $(y / x) \cdot x = y$
\item $0 \cdot x = x \cdot 0 = x$
\end{enumerate}
\end{definition}
Then the following holds:

\begin{theorem}[Chapter 7 of~\cite{FreeseMcKenzie-Commutator}] \label{theorem:loops}
Let $\algA$ be a nilpotent algebra with a Mal'cev term $m(x,y,z)$. For each $0 \in A$ the operation defined by $x \cdot y = m(x,0,y)$ is a loop multiplication with neutral element $0$. Also the left and right inverse operations $\backslash$ and $/$ can be defined as polynomials over $\algA$. \hfill $\square$
\end{theorem}

In other words every nilpotent algebra $\algA = (A, F)$ with Mal'cev term is polynomially equivalent to a nilpotent loop $(A,\cdot,\backslash,/)$ expanded by additional operations $F$. We denote this expansion by $(\algA,\cdot,\backslash,/)$. In order to further give a characterization of supernilpotent algebras we introduce the following notation: 

\begin{definition}
Let $f \in \Pol_m(\algA)$. We say $f(x_1,\ldots,x_m)$ \emph{absorbs} $(a_1,\ldots, a_m)$ to $a$ if $f(b_1,\ldots,b_m) = a$, whenever $b_i = a_i$ for some $i$. We say $f(x_1,\ldots,x_m)$ is \emph{$a$-absorbing} if $f$ absorbs $(a,a,\ldots,a)$ to $a$. 
\end{definition}

Then the following holds:

\begin{theorem}[Proposition 6.16. in \cite{AichingerMudrinski-highercomm}] \label{theorem:absorbing}
Let $\algA$ be an algebra with Mal'cev term and let $0 \in A$. Then $\algA$ is supernilpotent of degree $n$ if and only if every $0$-absorbing $c \in \Pol_{n} \algA$ is equivalent to $0$. \hfill $\square$
\end{theorem}

Theorems~\ref{theorem:absorbing} and~\ref{theorem:loops} were used in~\cite{WiresSupernilpotency} to give a canonical representation of polynomials in supernilpotent algebras with Mal'cev term. In the proof of our main result in the next section we will recapitulate Wires' proof and slightly refine it.

\section{Equation solvability in supernilpotent algebras with Mal'cev term} \label{section:supernilpotent}

In this section we show that the polynomial equation solvability problem $\pEq(\algA)$ is in P for supernilpotent $\algA$ with Mal'cev term. The main ingredient for this is Lemma~\ref{lemma:interpolation}, which states that computing the range of a polynomial expression $p(\bar x)$ over $\algA$ requires only to check substitutions of $\bar x$ for which the number of non $0$ entries is bounded by some constant $d$. In fact, we can show that this interpolation property is equivalent to supernilpotency for finite algebras with Mal'cev term.

We start with some basic observations. By Theorem~\ref{theorem:loops} we know that $\algA$ has polynomials that define loop operations $\cdot$, $\backslash$, $/$ and $0$. Clearly $\pEq(\algA)$ reduces to the equation solvability problem over the expanded algebra $(\algA, \cdot, \backslash, /, 0)$. Hence without loss of generality we can assume that $\algA$ contains the loop operations given by Theorem~\ref{theorem:loops}. (Note however, that we do not know a priori, whether $(\algA, \cdot, \backslash, /, 0)$ and $\algA$ have the same complexity up to polynomial time.)


In every algebra $\algA$ with loop operations, $f(\bar x) = g(\bar x)$ is equivalent to $f(\bar x)/g(\bar x) = 0$. Hence for both for the equation solvability problem and the identity checking problem we can restrict our input to equations of the form $f(x_1,\ldots,x_m) = 0$. 

Also note that $f(x_1,\ldots,x_m) = 0$ holds for all $x_1,\ldots,x_m$ if and only if none of the equations $f(x_1,\ldots,x_m)/a = 0$ for $0 \neq a \in A$ has a solution. Hence if $\algA$ is finite, the complement of $\pId(\algA)$ reduces to $\pEq(\algA)$ in polynomial time. It is however open if this holds for finite algebras in general, see also Problem 1 in~\cite{HorvathSzabo-groups}. It was already observed in~\cite{AichingerMudrinski-highercomm} that the polynomial identity checking problem is in P for supernilpotent algebras with Mal'cev term, hence our result can be seen as a strengthening of that.


We are going to stick to the following notation: Let us write $\prod_{i = 1}^n x_i$ or $x_1 \cdot x_2 \cdots x_n$ for the left associated product $( \cdots ((x_1 \cdot x_2) \cdot x_3) \cdots x_n)$ and let $x^n = \prod_{i = 1}^n x$. For $n \in \mathbb N$ let us write $[n] = \{1,\ldots, n\}$ and $\binom{[n]}{k} = \{S \subseteq [n] : |S| =k \}$. For a tuple $\bar x = (x_1, \ldots, x_n)$ of variables or elements of $\algA$ and $S \subseteq [n]$, let us write $\bar x_S$ for the $n$-tuple, where the $i$-th entry is equal to $x_i$ if $i \in S$ and $0$ otherwise, and let us write $\bar x_{\restriction S}$ for the $|S|$-tuple $(x_i)_{i \in S}$. If for instance $\bar x = (x_1,x_2,x_3,x_4,x_5)$ and $S = \{1,2,5\}$ then $\bar x_{S} = (x_1,x_2,0,0,x_5)$ and $\bar x_{\restriction S} = (x_1,x_2,x_5)$.

Our proof is going to rely on a representation of polynomials $f(x_1,\ldots, x_m) \in \Pol(\algA)$ as the product of $|S|$-ary $0$-absorbing terms $t_S(\bar x_{\restriction S})$ for all subsets $S \subseteq [m]$. That such representations exist was already known (see for instance Theorem 3.8 of \cite{WiresSupernilpotency}), but we are going to give a slightly more restrictive version that depends on a given enumeration of the elements of $A$ and a tuple $\bar b \in A^m$:

\begin{lemma} \label{lemma:canonicalrep}
Let $\algA$ be a finite nilpotent algebra with Mal'cev term, let $0 \in A$ and $f(x_1,\ldots,x_m) \in \Pol_m(\algA)$. Furthermore let $a_1, a_2, \ldots, a_N$ be an enumeration of the elements of $A$ and $\bar b \in A^m$. 
Then $f(x_1,\ldots,x_m)$ is equivalent to a polynomial of the form $\prod_{i = 0}^{m} r_i(x_1,\ldots,x_m)$, where the terms $r_i(x_1,\ldots,x_m)$ are given by the recursion $r_0(x_1,\ldots,x_m) = f(0,0,\ldots,0)$ and 
\begin{align}
t_S(\bar x_{\restriction S}) &= \left(\prod_{i = 0}^k r_i(\bar x_S)\right) \backslash f(\bar x_S) \text{ for } S \in \tbinom{[m]}{k+1}, \label{eq:tterms}\\
r_{k+1}(x_1,\ldots,x_m) &= \prod_{i=0}^{N} \prod_{\substack{S \in \binom{[m]}{k+1} \\t_S(\bar b_{\restriction S}) = a_i}} t_S(\bar x_{\restriction S}) \label{eq:rterms}, 
\end{align}
for $0 \leq k < m$. For all $S \subseteq [m]$ the $|S|$-ary terms $t_S(\bar x_{\restriction S})$ are $0$-absorbing.\\
If $\algA$ is moreover supernilpotent of degree $n$, then all terms $r_{k}(x_1,\ldots,x_m)$ for $k \geq n$ are equivalent to $0$.
\end{lemma}

Before we give the proof of Lemma \ref{lemma:canonicalrep} we would like to point out that this representation is not unique: depending on the ordering of the sets $S \in \binom{[m]}{k+1}$ with $t_S(\bar b_{\restriction S}) = a_i$ in (\ref{eq:rterms}) we might get different representations of $f(x_1,\ldots,x_m)$. However the main reason for us to show Lemma \ref{lemma:canonicalrep} is not to represent $f$ in a canonical way, but to show that for every $T \subseteq [m]$, $f(\bar b_T)$ is evaluated to a product of powers of the elements of $\algA$
$$f(\bar b_T)  = \prod_{i = 1}^{n-1} a_1^{\beta_{i,1}} \cdot a_2^{\beta_{i,2}} \cdots a_N^{\beta_{i,N}},$$
where $n$ is the degree of supernilpotency of $\algA$. This fact will be essential in the proof of Lemma \ref{lemma:interpolation}.

\begin{proof}[Proof of Lemma \ref{lemma:canonicalrep}]
We are going to prove the following: For every $0 \leq k \leq m$ and every $S \in \binom{[m]}{k}$ we have that $t_S(\bar x_{\restriction S})$ is $0$-absorbing and
\begin{equation}\label{eq:ind}f(\bar x_S) = \prod_{i = 0}^k r_i(\bar x_S).\end{equation}We prove the claim by induction on $k$.

For $k = 0$ we have $r_0(x_1,\ldots,x_m) = f(0,\ldots,0)$ and $t_\emptyset = 0$, which clearly satisfies the claim. So let us consider the induction step $k \to k+1$. We first show that for every $S \in \binom{[m]}{k+1}$ the term $t_S(\bar x_{\restriction S})$ is 0-absorbing: For that, let $S'$ be a proper subset of $S$. Then, by the definition of $t_S(\bar x_{\restriction S})$ in (\ref{eq:tterms}) and the induction hypothesis (\ref{eq:ind}) for $S'$ we have
$$t_S((\bar x_{\restriction S})_{S'}) = \left(\prod_{i = 0}^k r_i(\bar x_{S'})\right) \backslash f(\bar x_{S'}) = f(\bar x_{S'}) \backslash f(\bar x_{S'}) = 0.$$
Hence $t_S$ is $0$-absorbing for every $S \in \binom{[m]}{k+1}$. In order to show (\ref{eq:ind}) for $S$ note that if we evaluate $r_{k+1}$ at $\bar x_S$, all factors in (\ref{eq:rterms}) except for $t_S(\bar x_{\restriction S})$ are equivalent to $0$, since they are $0$-absorbing. Therefore
$$\prod_{i = 0}^{k+1} r_i(\bar x_S) = \left(\prod_{i = 0}^{k} r_i(\bar x_{\restriction S})\right) \cdot t_S(\bar x_{\restriction S}) = \left(\prod_{i = 0}^{k} r_i(\bar x_S)\right) \cdot \left(\left(\prod_{i = 0}^k r_i(\bar x_S)\right) \backslash f(\bar x_S)\right) = f(\bar x_S),$$
which proves the claim for $k+1$. Thus we proved our claim. The first part of the Lemma follows from (\ref{eq:ind}) for $m = k$.

If $\algA$ is moreover supernilpotent of degree $n$ all the terms $t_S(\bar x_{\restriction S})$ for $|S| \geq n$ are equivalent to $0$, since they are 0-absorbing (cf. Theorem~\ref{theorem:absorbing}). Therefore also all terms $r_{k}(x_1,\ldots,x_m)$ are equivalent to $0$ for $k \geq n$.
\end{proof}

We are going to use Lemma \ref{lemma:canonicalrep} together with the following iterated version of Ramsey's theorem to prove Lemma~\ref{lemma:interpolation}.

\begin{theorem}[Ramsey's theorem] \label{theorem:Ramsey}
Let $n,k$ and $l$ be positive integers. Then there exists a positive integer $d = d(n,k,l)$, such that for all sets $S$ with $|S| \geq d$ and for all $k$-colorings $\gamma$ of the $\leq n$-elements subsets of $S$, there exists $H \subseteq S$ with $|H|=l$ such that all subsets of $H$ of the same size have the same color.
\end{theorem}

\begin{lemma} \label{lemma:interpolation}
Let $\algA$ be a finite algebra with Mal'cev term that is supernilpotent of degree $n$ and let $0 \in A$. Then there exists a positive integer $d = d(\algA)$ such that for every $m \geq d$, for every polynomial $f \in \Pol_m(\algA)$ and for every $\bar b = (b_1, \ldots, b_m) \in A^m$ there exists a set $T \in \binom{[m]}{d}$ with $f(\bar b_T) = f(\bar b)$.
\end{lemma}

\begin{proof}
We follow the proof steps of the analogous result for nilpotent groups in \cite[Lemma 3.1]{HorvathEqSolvNilp}. Let $e$ be the \emph{exponent} of $\algA$, i.e. the smallest positive integer such that $x^e = 0$ for all $x \in A$ and let $l = e \cdot (n-1)!$ and $k = e^{n \cdot |A|}$. We then claim that the Ramsey number $d = d(n-1,k,l)$ given by Theorem \ref{theorem:Ramsey} satisfies the Lemma.

First recall the representation result in Lemma \ref{lemma:canonicalrep}. For a given enumeration $a_1, \ldots a_{N}$ of the elements of $\algA$ it gives us
$$f(\bar b)  = \prod_{i = 1}^{n-1} a_1^{\alpha_{i,1}} \cdot a_2^{\alpha_{i,2}} \cdot \cdots \cdot a_N^{\alpha_{i,N}},$$
where $\alpha_{i,j}$ is the number of sets $S \in \binom{[m]}{i}$, such that $t_S(\bar b_{\restriction S}) = a_j$. Let $H$ be an arbitrary subset of $[m]$ and $H^c = [m] \setminus H$. Since all of the terms $t_S(\bar x_{\restriction S})$ are $0$-absorbing, we have the same representation for $f(\bar b_{H^c})$, i.e.
$$f(\bar b_{H^c})  = \prod_{i = 1}^{n-1} a_1^{\beta_{i,1}} \cdot a_2^{\beta_{i,2}} \cdot \cdots \cdot a_N^{\beta_{i,N}},$$
where $\beta_{i,j}$ is the number of sets $S \in \binom{[m]}{i}$, $S \subseteq H^c$, such that $t_S(\bar b_{\restriction S}) = a_j$. 

We claim that there is a non-empty set $H$ such that the corresponding exponents $\beta_{i,j}$ are equal to $\alpha_{i,j}$ modulo $e$. If this is the case, we have that $f(\bar b) = f(\bar b_{H^c})$. If $|H^c| \leq d$, we set $T = H^c$ and are done. Otherwise we can find such $T$ by iterating the procedure for the $|H^c|$-ary polynomial defined by $f(\bar x_{H^c})$. Hence it only remains to prove this claim.

For every set $I \subseteq [m]$ let $\gamma_{i,j}(I)$ denote the number of all set $S \in \binom{[m]}{i}$ such that $t_S(\bar b_{\restriction S}) = a_j$ and $I \subseteq S$. By the inclusion-exclusion principle we have for a given $H$ that
$$\alpha_{i,j}-\beta_{i,j} = -\sum_{\substack{\emptyset \neq I \subseteq H \\ |I| \leq i}} (-1)^{|I|} \gamma_{i,j}(I) = -\sum_{\substack{\emptyset \neq I \subseteq H \\ |I| < n}} (-1)^{|I|} \gamma_{i,j}(I).$$
We show that there is an $H$ such that all summands of the form $\sum_{I \subseteq H, |I| = s} \gamma_{i,j}(I)$ for $s < n$ are divisible by $e$. The function $\gamma(I) := (\gamma_{i,j}(I))_{i \in [n], j \in [N]}$ is a coloring of subsets of $[m]$ with $k$ colors. By Theorem \ref{theorem:Ramsey} there is an $H$ with $|H| = l$ such that $\gamma$ is monochromatic on all subsets of size at most $n-1$ of $H$. This implies that $\binom{l}{s}$ divides $\sum_{I \subseteq H, |I| = s} \gamma_{i,j}(I)$ for every $s < n$. Since $l = e \cdot (n-1)!$, we know that $\binom{l}{s}$ is divisible by $e$ for every $s < n$. Hence also $\alpha_{i,j} - \beta_{i,j}$ is divisible by $e$, which concludes the proof.
\end{proof}

We remark that Lemma~\ref{lemma:interpolation} gives us a characterization of finite supernilpotent algebra with Mal'cev term:

\begin{corollary}
Let $\algA$ be a finite algebra with Mal'cev term and $0 \in A$. Then $\algA$ is supernilpotent if and only if there is a positive integer $d$ such that for every polynomial $f(x_1,\ldots,x_m) \in \Pol(\algA)$ and every tuple $\bar r \in A^m$ there is an index set $T \subseteq [m]$ with $|T| \leq d$ and $f(\bar r) = f(\bar r_T)$.
\end{corollary}

\begin{proof}
It only remains to show that if $\algA$ is not supernipotent, it does not have the interpolation property described above. By Theorem~\ref{theorem:absorbing} for every $d \in \mathbb N$ there is a $0$-absorbing polynomial $f(x_1,\ldots,x_d,x_{d+1})$ and a tuple $\bar r \in A^{d+1}$ such that $f(\bar r) \neq 0$. But as $f$ is $0$-absorbing, $f(\bar r_T) = 0$ holds for every $T \subseteq [d+1]$, $T \neq [d+1]$.
\end{proof}

Lemma \ref{lemma:interpolation} now implies our main result:

\begin{theorem} \label{theorem:supernilpotent}
Let $\algA$ be a finite supernilpotent algebra with Mal'cev term. Then the equation solvability problem for $\algA$ can be decided in polynomial time.
\end{theorem}

\begin{proof}
By the discussion at the beginning of this section we only have to consider equations of the form $f(x_1,\ldots,x_m) = 0$ as input. By Lemma \ref{lemma:interpolation}, there is a solution $\bar b$ with $f(\bar b) = 0$ if and only if $f(\bar b_T) = 0$ for some $T$ with $|T| = \min(d,m)$, where $d$ is a constant only depending on $\algA$. For $m \geq d$ there are $|A|^d \cdot \binom{m}{d} = \mathcal{O}(m^d)$ many tuples of the form $\bar b_T$. Hence evaluating $f$ on all those tuples and checking whether the result is $0$ takes polynomial time $\mathcal{O}(l(f(\bar x))^d)$ and yields whether $f(x_1,\ldots,x_m) = 0$ is solvable.
\end{proof}

We remark that due to the use of Ramsey's theorem the value of $d$ in Lemma~\ref{lemma:interpolation} might be very large; we can only obtain upper bounds that are superexponential in $|\algA|$. The best known algorithm for nilpotent groups $G$ runs in polynomial with exponent ${\frac{1}{2} |G|^2 \log(|G|)}$ and is due to F\"oldv\'ari~\cite{Attila-nilpotentEqs}. This indicates that our algorithm in Theorem \ref{theorem:supernilpotent} might be far from being optimal.

\section{Nilpotent algebras with hard equation solvability and identity checking problems} \label{section:nilpotent}

For every prime $p$ let $\algA_p = (\Z_{p^2},+,0,-,(f_n)_{n \in \mathbb N})$ be the cyclic group of order $p^2$, together with the $n$-ary operations $f_n(x_1,\ldots,x_n) = p \cdot x_1 \cdot x_2 \cdots x_n$ for every $n \in \mathbb N$. In this section we are going to show that $\pEq(\algA_p)$ is NP-complete and $\pId(\algA_p)$ is co-NP-complete for every $p > 2$.

It is easy to see that $\algA_p$ is nilpotent of degree 2 for every $p$; the equivalence classes of $[\mathbf{1}_{\algA_p},\mathbf{1}_{\algA_p}]$ are exactly the cosets of $\Z_p$. Furthermore it follows straightforward from Theorem~\ref{theorem:absorbing} that $\algA_p$ is not supernilpotent, since every function $f_n(x_1,\ldots,x_n)$ is $0$-absorbing but not equivalent to $0$. However we remark that every restriction of $\algA_p$ to finitely many of its operators gives us a supernilpotent algebra.

It is a priori not clear how to encode the input when phrasing $\pId(\algA_p)$ as computational problem, since $\algA_p$ is of infinite type. However, in every arity there is exactly one operation $f_n$ so one can still find a reasonable such encoding of terms, i.e. one where the size of an input polynomial $f(\bar x)$ is linear in its length $l(f(\bar x))$. With respect to such an encoding the following holds:

\begin{theorem} \label{theorem:hardness}
Let $p$ be an odd prime and let $\algA_p = (\Z_{p^2},+,0,-,(f_n)_{n \in \mathbb N})$ with $f_n(x_1,\ldots,x_n) = p \cdot x_1 \cdot x_2 \cdots x_n$. Then the equation solvability problem $\pEq(\algA_p)$ is NP-complete and the identity checking problem $\pId(\algA_p)$ is co-NP-complete.
\end{theorem}

\begin{proof}
We prove that $\pEq(\algA_p)$ is NP-complete by reducing the graph $p$-colorability problem to it. To do so, for every instance of a graph $G = (V,E)$ we define the term:
$$t_G((x_v)_{v \in V}) = f_{(p-1)\cdot|E|}\left((x_{v_1} - x_{v_2})_{\substack{(v_1,v_2) \in E \\ i \in [p-1]}} \right) = p \cdot \prod_{(v_1,v_2) \in E}{(x_{v_1} - x_{v_2})^{p-1}}$$ 
Note that the order of edges is irrelevant and for a tuple $(r_v)_{v \in V}$ in $\Z_{p^2}$ the value of $t_G((r_v)_{v \in V})$ only depends on the cosets of $r_v$ with respect to $\Z_p$. Moreover $t_G((r_v)_{v \in V}) = 0$ holds if and only if there is an edge $(v_1,v_2) \in E$ such that $r_{v_1}$ and $r_{v_2}$ are in the same coset of $\Z_p$; otherwise $t_G((r_v)_{v \in V}) = p$.

Thus, if the equation $t_G((x_v)_{v \in V}) = p$ has a solution $(r_v)_{v \in V}$, then the coloring that assigns to each vertex $v$ the color $r_v \Z_p$ is a proper coloring of the graph $G$ with $p$ colors. Conversely every coloring of $G$ with $p$ many colors induces a solution of the equation (by assigning to every color a unique coset of $\Z_p$). Thus $p$-colorability reduces to $\pEq(\algA_p)$, which consequently is NP-complete.

Analogously a graph is not $p$-colorable if and only if $t_G((x_v)_{v \in V}) = 0$ holds for all values of $(x_v)_{v \in V}$. Thus $\pId(\algA_p)$ is coNP-complete.
\end{proof}

We conclude with the question, whether this hardness result fits into a bigger context. By~\cite{HorvathSzabo-extendedgroups} and~\cite{Horvathetal-nonsolvablegroups} every non-nilpotent group has a polynomial extension, whose identity checking problem is co-NP-complete. By~\cite{BurrisLawrence-rings} also for rings this statement is true. Therefore we ask:

\begin{question} \label{question:hardness}
Does every non-supernilpotent finite algebra with Mal'cev term have a polynomial extension, whose 
\begin{itemize}
\item identity checking problem is co-NP-complete?
\item equation solvability problem is NP-complete?
\end{itemize}
\end{question}

A first step in answering Question~\ref{question:hardness} would be to study the question for nilpotent, but not supernilpotent algebras. In this case we have much structural information to work with, due to Theorem~\ref{theorem:loops} and Theorem~\ref{theorem:absorbing}. Note that Question~\ref{question:hardness} might have different answers, depending on the encoding of the input (see also the discussion in Section~\ref{sect:eqsolv}), and also depending on whether we restrict ourselves to algebras of finite type or not, as in our example.\\

\textbf{Recent progress:} After the submission of this article it came to the authors attention that Idziak and Krzaczkowski independently proved Theorem \ref{theorem:supernilpotent} in their paper \cite{IdziakKrzaczkowski}, where they studied the equation solvability problem and the identity checking problem in the more general setting of algebras from congruence modular varieties. Moreover they proved several hardness results that partially answer Question \ref{question:hardness}: By their work, every non-solvable algebra $\algA$ with a Mal'cev term has polynomial extensions with hard $\pEq$ and $\pId$ problems. Furthermore every solvable, but non-nilpotent  algebra $\algA$ with Mal'cev term has a quotient for which Question \ref{question:hardness} has a positive answer. However in the nilpotent, but not supernilpotent case, the situation seems to be more complicated \cite{IdziakKrzaczkowski_private}: There are 2-nilpotent, but non-supernilpotent algebras of finite type such that every extension of it by finitely many polynomials has tractable equation solvability and identity checking problem (even if the input polynomials are encoded by circuits). However an extension of these algebras by \emph{infinitely many} polynomials induced hardness of both problems as in the example of Theorem \ref{theorem:hardness}.

\section*{Acknowledgments}
The author would like to thank Jakub Op\v{r}sal for introducing him to (higher) commutators and giving several other helpful remarks.

\bibliographystyle{abbrv}
\bibliography{kompatscher_equation_solvability}

\end{document}